\documentclass[11pt, reqno]{amsart}
\usepackage{appendix}
\usepackage{xcolor}
\usepackage[utf8]{inputenc}
\usepackage[colorlinks=true,hyperindex,pagebackref,linktocpage=true]{hyperref}
\usepackage{cleveref}
\usepackage{mathtools}
\usepackage{amsfonts}
\usepackage{amssymb}
\usepackage{tikz-cd}
\usepackage[T1]{fontenc}

\hypersetup{
	colorlinks,
	linkcolor={violet!60!black},
	citecolor={cyan!60!black},
	urlcolor={orange!60!black}
}
\usepackage{stmaryrd}
\usepackage{mathrsfs}  
\usepackage{varwidth}
\theoremstyle{plain}
\newtheorem{theorem}{Theorem}[section]
\newtheorem{proposition}[theorem]{Proposition}
\newtheorem{lemma}[theorem]{Lemma}
\newtheorem{corollary}[theorem]{Corollary}

\theoremstyle{definition}
\newtheorem{definition}[theorem]{Definition}

\newtheorem{remark}[theorem]{Remark}
\newtheorem{convention}[theorem]{Convention}

\newcommand{\nc}{\newcommand}
\nc{\on}{\operatorname}

\nc{\Q}{\mathbb{Q}}
\nc{\Z}{\mathbb{Z}}
\nc{\cl}{\mathrm{cl}}

\nc{\fraka}{{\mathfrak a}} \nc{\bba}{{\mathbf a}}
\nc{\frakb}{{\mathfrak b}}
\nc{\frakc}{{\mathfrak c}}
\nc{\frakd}{{\mathfrak d}}
\nc{\frake}{{\mathfrak e}}
\nc{\frakf}{{\mathfrak f}}
\nc{\frakg}{{\mathfrak g}}
\nc{\frakh}{{\mathfrak h}}
\nc{\fraki}{{\mathfrak i}}
\nc{\frakj}{{\mathfrak j}}
\nc{\frakk}{{\mathfrak k}}
\nc{\frakl}{{\mathfrak l}}
\nc{\frakm}{{\mathfrak m}}
\nc{\frakn}{{\mathfrak n}}
\nc{\frako}{{\mathfrak o}}
\nc{\frakp}{{\mathfrak p}}
\nc{\frakq}{{\mathfrak q}}
\nc{\frakr}{{\mathfrak r}}
\nc{\fraks}{{\mathfrak s}}
\nc{\frakt}{{\mathfrak t}}
\nc{\fraku}{{\mathfrak u}}
\nc{\frakv}{{\mathfrak v}}
\nc{\frakw}{{\mathfrak w}}
\nc{\frakx}{{\mathfrak x}}
\nc{\fraky}{{\mathfrak y}}
\nc{\frakz}{{\mathfrak z}}
\nc{\frakA}{{\mathfrak A}}
\nc{\frakB}{{\mathfrak B}}
\nc{\frakC}{{\mathfrak C}}
\nc{\frakD}{{\mathfrak D}}
\nc{\frakE}{{\mathfrak E}}
\nc{\frakF}{{\mathfrak F}}
\nc{\frakG}{{\mathfrak G}}
\nc{\frakH}{{\mathfrak H}}
\nc{\frakI}{{\mathfrak I}}
\nc{\frakJ}{{\mathfrak J}}
\nc{\frakK}{{\mathfrak K}}
\nc{\frakL}{{\mathfrak L}}
\nc{\frakM}{{\mathfrak M}}
\nc{\frakN}{{\mathfrak N}}
\nc{\frakO}{{\mathfrak O}}
\nc{\frakP}{{\mathfrak P}}
\nc{\frakQ}{{\mathfrak Q}}
\nc{\frakR}{{\mathfrak R}}
\nc{\frakS}{{\mathfrak S}}
\nc{\frakT}{{\mathfrak T}}
\nc{\frakU}{{\mathfrak U}}
\nc{\frakV}{{\mathfrak V}}
\nc{\frakW}{{\mathfrak W}}
\nc{\frakX}{{\mathfrak X}}
\nc{\frakY}{{\mathfrak Y}}
\nc{\frakZ}{{\mathfrak Z}}
\nc{\bbA}{{\mathbb A}}
\nc{\bbB}{{\mathbb B}}
\nc{\bbC}{{\mathbb C}}
\nc{\bbD}{{\mathbb D}}
\nc{\bbE}{{\mathbb E}}
\nc{\bbF}{{\mathbb F}} \nc{\bbf}{{\mathbf f}}
\nc{\bbG}{{\mathbb G}}
\nc{\bbH}{{\mathbb H}}
\nc{\bbI}{{\mathbb I}}
\nc{\bbJ}{{\mathbb J}}
\nc{\bbK}{{\mathbb K}}
\nc{\bbL}{{\mathbb L}}
\nc{\bbM}{{\mathbb M}}
\nc{\bbN}{{\mathbb N}}
\nc{\bbO}{{\mathbb O}}
\nc{\bbP}{{\mathbb P}}
\nc{\bbQ}{{\mathbb Q}}
\nc{\bbR}{{\mathbb R}}
\nc{\bbS}{{\mathbb S}}
\nc{\bbT}{{\mathbb T}}
\nc{\bbU}{{\mathbb U}}
\nc{\bbV}{{\mathbb V}}
\nc{\bbW}{{\mathbb W}}
\nc{\bbX}{{\mathbb X}}
\nc{\bbY}{{\mathbb Y}}
\nc{\bbZ}{{\mathbb Z}}
\nc{\calA}{{\mathcal A}}
\nc{\calB}{{\mathcal B}}
\nc{\calC}{{\mathcal C}}
\nc{\calD}{{\mathcal D}}
\nc{\calE}{{\mathcal E}}
\nc{\calF}{{\mathcal F}}
\nc{\calG}{{\mathcal G}}
\nc{\calH}{{\mathcal H}}
\nc{\calI}{{\mathcal I}}
\nc{\calJ}{{\mathcal J}}
\nc{\calK}{{\mathcal K}}
\nc{\calL}{{\mathcal L}}
\nc{\calM}{{\mathcal M}}
\nc{\calN}{{\mathcal N}}
\nc{\calO}{{\mathcal O}}
\nc{\calP}{{\mathcal P}}
\nc{\calQ}{{\mathcal Q}}
\nc{\calR}{{\mathcal R}}
\nc{\calS}{{\mathcal S}}
\nc{\calT}{{\mathcal T}}
\nc{\calU}{{\mathcal U}}
\nc{\calV}{{\mathcal V}}
\nc{\calW}{{\mathcal W}}
\nc{\calX}{{\mathcal X}}
\nc{\calY}{{\mathcal Y}}
\nc{\calZ}{{\mathcal Z}}

\nc{\scrA}{{\mathscr A}}
\nc{\scrB}{{\mathscr B}}
\nc{\scrC}{{\mathscr C}}
\nc{\scrD}{{\mathscr D}}
\nc{\scrE}{{\mathscr E}}
\nc{\scrF}{{\mathscr F}}
\nc{\scrG}{{\mathscr G}}
\nc{\scrH}{{\mathscr H}}
\nc{\scrI}{{\mathscr J}}
\nc{\scrJ}{{\mathscr I}}
\nc{\scrK}{{\mathscr K}}
\nc{\scrL}{{\mathscr L}}
\nc{\scrM}{{\mathscr M}}
\nc{\scrN}{{\mathscr N}}
\nc{\scrO}{{\mathscr O}}
\nc{\scrP}{{\mathscr P}}
\nc{\scrQ}{{\mathscr Q}}
\nc{\scrR}{{\mathscr R}}
\nc{\D}{{\on{D}}}
\nc{\Div}{{\on{Div}}}
\nc{\Perv}{{\on{Perv}}}

\nc{\bnu}{{\bar{ \nu}}}
\nc{\olO}{\bar{\calO}}

\nc{\al}{{\alpha}} 
\nc{\be}{{\beta}}
\nc{\ga}{{\gamma}} \nc{\Ga}{{\Gamma}}
\nc{\hGa}{\hat{\Gamma}}
\nc{\ve}{{\varepsilon}} 
\nc{\la}{{\lambda}} \nc{\La}{{\Lambda}}
\nc{\om}{\omega} \nc{\Om}{\Omega} 
\nc{\sig}{{\sigma}} \nc{\Sig}{{\Sigma}}
\nc{\dR}{{\mathrm{dR}}}
\nc{\Perf}{{\mathrm{Perf}}}
\nc{\Gm}{{\mathbb{G}_m}}
\nc{\colim}{{\on{colim}}}
\nc{\et}{\mathrm{\acute{e}t}}

\makeatletter
\DeclareFontEncoding{LS1}{}{}
\DeclareFontSubstitution{LS1}{stix}{m}{n}
\DeclareMathAlphabet{\rhomalpha}{LS1}{stixscr}{m}{n}
\makeatother

\nc{\Spa}{\on{{Spa}}}
\nc{\Spd}{\on{{Spd}}}
\nc{\tnb}{\psi_{\rm tame}}
\nc{\oM}{\overline{{M}}}
\nc{\op}{{\on{op}}}
\nc{\ad}{{\on{ad}}}
\nc{\alg}{{\on{alg}}}
\nc{\Ad}{{\on{Ad}}}
\nc{\Adm}{{\on{Adm}}} \nc{\aff}{{\on{af}}}
\nc{\Aut}{{\on{Aut}}}
\nc{\Bun}{{\on{Bun}}}
\nc{\cha}{{\on{char}}}
\nc{\der}{{\on{der}}}
\nc{\Der}{{\on{Der}}}
\nc{\diag}{{\on{diag}}}
\nc{\End}{{\on{End}}}
\nc{\Fl}{{\calF\!\ell}}
\nc{\Tr}{{\on{Transp}}}
\nc{\TR}{{\calT\!\calR}}
\nc{\Gal}{{\on{Gal}}}
\nc{\Gr}{{\on{Gr}}}
\nc{\Hk}{{\on{Hk}}}
\nc{\rH}{{\on{H}}}
\nc{\Hom}{{\on{Hom}}}
\nc{\IC}{{\on{IC}}}
\nc{\id}{{\on{id}}}
\nc{\Id}{{\on{Id}}}
\nc{\ind}{{\on{ind}}}
\nc{\Ind}{{\on{Ind}}}
\nc{\Lie}{{\on{Lie}}}
\nc{\Pic}{{\on{Pic}}}
\nc{\pr}{{\on{pr}}}
\nc{\Res}{{\on{Res}}}
\nc{\res}{{\on{res}}} \nc{\Sat}{{\on{Sat}}}
\nc{\spc}{{\on{sc}}}
\nc{\drv}{{\on{der}}}
\nc{\sgn}{{\on{sgn}}}
\nc{\Spec}{{\on{Spec}}}\nc{\Spf}{\on{Spf}} 
\nc{\Sph}{\on{Sph}}
\nc{\St}{{\on{St}}}
\nc{\tr}{{\on{tr}}}
\nc{\Mod}{{\mathrm{-Mod}}}
\nc{\Hilb}{{\on{Hilb}}} 
\nc{\Ext}{{\on{Ext}}} 
\nc{\vs}{{\on{Vec}}}
\nc{\ev}{{\on{ev}}}
\nc{\nO}{{\breve{\calO}}}
\nc{\tS}{{\tilde{S}}}
\nc{\spe}{{\on{sp}}}
\nc{\loc}{{\on{loc}}}
\nc{\pre}{{\on{pre}}}

\nc{\dimt}{{\on{dim.trg}}}

\nc{\co}{\colon}
\nc{\dia}{{\diamondsuit}}

\nc{\nscrR}{{\mathscr{R}^{\on{nr}}}}

\nc{\GL}{{\on{GL}}}

\nc{\Gl}{\on{Gl}} 
\nc{\GSp}{{\on{GSp}}}
\nc{\gl}{{\frakg\frakl}}
\nc{\SL}{{\on{SL}}} 
\nc{\SU}{{\on{SU}}} 
\nc{\SO}{{\on{SO}}}
\nc{\PGL}{{\on{PGL}}}

\nc{\Conv}{{\on{Conv}}}
\nc{\Rep}{{\on{Rep}}}
\nc{\Dom}{{\on{Dom}}}
\nc{\red}{{\on{red}}}
\nc{\act}{{\on{act}}}
\nc{\nr}{{\on{nr}}}
\nc{\ctf}{{\on{ctf}}}

\nc{\str}{{\on{-}}} 
\nc{\os}{{\bar{s}}}
\nc{\oeta}{{\bar{\eta}}}

\nc{\hookto}{\hookrightarrow}
\nc{\longto}{\longrightarrow}
\nc{\leftto}{\leftarrow}
\nc{\onto}{\twoheadrightarrow}
\nc{\lonto}{\twoheadleftarrow}

\nc{\pot}[1]{ [\hspace{-0,5mm}[ {#1} ]\hspace{-0,5mm}] }
\nc{\rpot}[1]{ (\hspace{-0,7mm}( {#1} )\hspace{-0,7mm}) }
\nc{\smallpot}{{ <\hspace{-1,0mm}<}}

\numberwithin{equation}{section}

\begin{document}
	
	\title{On the connectedness of $p$-adic period domains.}
	
	\author[I. Gleason, J. Louren\c{c}o]{Ian Gleason, Jo\~ao Louren\c{c}o}

	\address{Mathematisches Institut der Universit\"at Bonn, Endenicher Allee 60, Bonn, Germany}
	\email{igleason@uni-bonn.de}

	\address{Mathematisches Institut, Universität Münster, Einsteinstrasse 62, Münster, Germany}
	\email{j.lourenco@uni-muenster.de}

	\begin{abstract}
		We prove that all $p$-adic period domains (and their non-minuscule analogues) are geometrically connected.
		This answers a question of Hartl \cite{Har13} and has consequences to the geometry of Shimura and local Shimura varieties.
	\end{abstract}

	\maketitle
	\tableofcontents
	
	\section{Introduction}
	Period domains and their geometric properties are recurring themes in analytic geometry when studying Shimura varieties and their $p$-adic uniformization. 
	They are analytic open subsets of flag varieties of reductive groups arising as the image of the Grothendieck--Messing period morphism, which stems from the theory of $p$-divisible groups. 
	The first instance of period domains in the literature is due to Drinfeld \cite{Drinf76}, who introduced the Drinfeld upper half-space $\Omega_n$, and was later complemented by Gross--Hopkins \cite{Hopk94}, who treated the period morphism for the Lubin--Tate tower.
    However, the first rigorous definition of $p$-adic period domains in terms of weakly admissible and admissible loci was given in the seminal book of Rapoport--Zink \cite{RZ96}, which initiated their systematic study. Since then, additional significant contributions to the field include the works of Hartl \cite{Hartl08}, Rapoport--Viehmann \cite{RV14}, Scholze--Weinstein \cite{SW13, SW20}, and Chen--Fargues--Shen \cite{MFS21}. We also refer to the book of Dat--Orlik--Rapoport \cite{DOR10} for a detailed introduction to the subject replete with examples.
    
	The purpose of this article is to prove that $p$-adic period domains are geometrically connected. Our results answers a conjecture of Hartl, see \cite[Conjecture 6.5]{Har13}.
	It is also a key ingredient for $p$-adic uniformization of Newton strata on Shimura varieties. 
	For a long time, it was common in the literature to assume that the coweight $\mu$ bounding the $p$-adic shtukas is minuscule, see for instance Rapoport--Zink \cite{RZ96} or Rapoport--Viehmann \cite{RV14}, because otherwise there was no hope of obtaining a moduli space representable by rigid-analytic spaces. We do not make any such assumption in this paper, since we can work entirely within the theory of diamonds \cite{Sch17}, following Scholze--Weinstein \cite{SW20}.  \\

	We consider a $p$-adic shtuka datum $(G,b,\mu)$ in the sense of Rapoport--Viehmann \cite[Definition 5.1]{RV14} but dropping the minuscule assumption on $\mu$, compare with \cite[Definition 23.1.1]{SW20}. This consists of a reductive group $G$ over $\bbQ_p$, an element $b$ of the Kottwitz set $B(G)=G(\breve \bbQ_p)/\mathrm{ad}_{\varphi}(G(\breve \bbQ_p))$ in the sense of Kottwitz \cite{Kot85}, and a geometric conjugacy class of (not necessarily minuscule) cocharacters $\mu \in \mathrm{Hom}(\bbG_m,G_{\bar{\bbQ}_p})/\mathrm{ad}(G(\bar \bbQ_p))$, such that $b\in B(G,\mu)$. 
	Let $E$ over $\bbQ_p$ be the reflex field of $\mu$, i.e.~ the finite field extension over which the conjugacy class of $\mu$ is defined.
	We let $\bbC_p$ denote a completed algebraic closure of $\bbQ_p$, $\breve{E}\subseteq \bbC_p$ denote the compositum of $E$ and $\breve{\bbQ}_p$ in $\bbC_p$, and $\Gamma$ denote the absolute Galois group of $\bbQ_p$.

	Given a characteristic $p$ perfectoid space $S$, one can construct functorially a $G$-bundle over the relative Fargues--Fontaine curve $X_{\mathrm{FF},S}$ which we denote by $\calE_{b}$. 
	Attached to $(G,\mu)$ one can define a spatial diamond $\Gr_{G,\mu}$ over $\Spd \breve{E}$ that parametrizes $B^+_{\on{dR}}$-lattices with $G$-structure bounded by $\mu$ in the Bruhat order \cite[\S\S19-22]{SW20}. 
	Moreover, using Beauville--Laszlo descent one can identify $\Gr_{G,\mu}$ with the moduli space of $G$-bundle modifications of $\calE_b$ 
	\begin{equation*}
		\Gr_{G,\mu}(S)=	\{(\calE,f) \mid f:\calE \dashrightarrow \calE_b,\, \on{rel}(f)\leq \mu\}/\cong	
	\end{equation*}
	whose relative position is bounded by $\mu$.
	This gives a Beauville--Laszlo uniformization map:
	\begin{align*}
		\calB\calL_b:\Gr_{G,\mu}& \to \Bun_G \\
	(\calE,f)& \mapsto \calE.	
	\end{align*}
	Here $\Bun_G$ denotes the small v-stack of $G$-bundles on the Fargues--Fontaine curve as in the book of Fargues--Scholze \cite{FS21}.
	Let $\Bun_G^1$ denote the sub-v-stack of $\Bun_G$ of those $G$-bundles that are fiberwise trivial \cite[\S III.2.3]{FS21}. By \cite[Corollary 22.5.1, Proposition 24.1.2]{SW20}, the $b$-admissible locus, $\Gr^b_{G,\mu}:=\calB\calL_b^{-1}(\Bun_G^1)$, is non-empty and open in $\Gr_{G,\mu}$.\footnote{We warn the reader that in some literature $\Gr^b_{G,\mu}$ denotes $\calB\calL_1^{-1}(\Bun_G^b)$ instead.} 

	When $\mu$ is minuscule and $G$ is quasi-split we have an identification $\Gr_{G,\mu}=(G/P_\mu)^\diamondsuit$, where $P_\mu$ is the parabolic subgroup defined by $\mu$. In this case, $\Gr_{G,\mu}$ is (the diamond attached to) a generalized flag variety (see \cite[\S 2.2]{AGLR22} for a discussion of the diamond functor). 
	Moreover, we also have a formula:
	\begin{equation*}
		\Gr_{G,\mu}^b=\pi_{\on{GM}}(\calM^\diamondsuit_{(G,b,\mu)})	
	\end{equation*}
	Where $\calM_{(G,b,\mu)}$ is the local Shimura variety attached to $(G,b,\mu)$ and $\pi_{\on{GM}}$ is the Grothendieck--Messing period morphism \cite{RV14, SW20}. 
	By \cite[Lemma 15.6]{Sch17} $\Gr^b_{G,\mu}$ is the diamond associated to a unique analytic open subset of $G/P_\mu$ that we denote by $\calF(G,b,\mu)^a$. This open subset is the $p$-adic period domain associated to $(G,b,\mu)$, and $\Gr^b_{G,\mu}=\calF(G,b,\mu)^{a,\diamondsuit}$.
	Our main theorem is the following:
	\begin{theorem}
		\label{mainthm}
		The map $\Gr_{G,\mu}^b\to \Spd\breve{E}$ has connected geometric fibers. Moreover, $\Gr_{G,\mu}^b\subset \Gr_{G,\mu}$ is geometrically dense as spaces over $\Spd \breve{E}$.
	\end{theorem}
	Let us put \Cref{mainthm} in context. In \cite{Kisin-mod-p-points} Kisin uses in an essential way the connected components of affine Deligne--Lusztig varieties (ADLV) to study the Langlands--Rapoport conjecture for integral models of Shimura varieties \cite{LR87}. 
	On the other hand, in \cite{Chen} Chen uses the connected components of ADLV to derive her main results on connected components of local Shimura varieties (LSV). 
	These two works motivated Chen--Kisin--Viehmann \cite{CKV15} to compute the connected components of ADLV at hyperspecial parahoric level building on previous work of Viehmann \cite{Vie08}.  
	Since then, several authors have pushed the strategy of \cite{CKV15} to compute connected components of ADLV deriving as corollaries results on the geometry of integral models of Shimura varieties (see the following results of Nie \cite[Theorem 1.1]{Nie}, He--Zhou \cite[Theorem 0.1]{HZ20}, Hamacher \cite[Theorem 1.1(3)]{Ham20}, Nie \cite[Theorem 0.2]{Nie21}). \\

	Now, Chen proves and uses a version of \Cref{mainthm} for period domains that arise from unramified Rapoport--Zink data as a key stepping stone to derive the main results in her work. This is where the connected components of ADLV enter in her argument.
	In \cite{gleason2022connected}, the first author together with Lim and Xu show that Chen's reasoning can be reversed, and use \Cref{mainthm} to compute the connected components of ADLV and the connected components of LSV \cite{gleason2022connected}. 

	Let us fix some notation. 
	Let $\calI$ denote an Iwahori group scheme over $\bbZ_p$ with generic fiber $G$. 
	Let $\varphi$ denote the canonical lift of arithmetic Frobenius to $\breve{\bbZ}_p$.
	Let $\on{Adm}(\mu)\subseteq \calI(\breve{\bbZ}_p)\backslash G(\breve{\bbQ}_p)/\calI(\breve{Z}_p)$ denote the $\mu$-admissible set of Kottwitz--Rapoport \cite{KR00}. 
	Let
	\begin{equation}X_\mu(b)=\{g\calI(\breve{\bbZ}_p)\mid g^{-1}b\varphi(g)\in \calI(\breve{\bbZ}_p)  \on{Adm}(\mu) \calI(\breve{\bbZ}_p) \}. \end{equation}
	This is the closed affine Deligne--Lusztig variety attached to $(G,b,\mu,\calI)$.
	It admits the structure of a perfect scheme locally perfectly of finite presentation \cite{Zhu17}, \cite{bhatt_scholze_projectivity_of_the_witt_vector_affine_grassmannian}.
	Let $\kappa_G:G(\breve{\bbQ}_p)\to \pi_1(G)_I$ denote the Kottwitz map \cite[7.4]{KottwitzII}.
	The map $\kappa_G$ induces a map $\omega_G:\pi_0(X_\mu(b))\to \pi_1(G)_I$ that factors through a unique coset $c_{b,\mu}\pi_1(G)_I^\varphi\in \pi_1(G)_I/\pi_1(G)_I^\varphi$.
	Here is an interesting consequence of our main theorem.
	\begin{corollary}
		\label{maincorollary}
		The Kottwitz map induces a bijection 
		\begin{equation}
			\omega_G:\pi_0(X^{\calK_p}_\mu(b))\xrightarrow{\cong} c_{b,\mu}\pi_1(G)_I^\varphi,
	\end{equation}
whenever $(b,\mu)$ is HN-irreducible. 
	\end{corollary}
	Our work, together with \cite{gleason2022connected}, finishes the problem of computing connected components of ADLV in mixed characteristic.\\ 

	Let us sketch the proof of our main theorem in the case where $G$ is quasi-split. Fix a Borel $B\subseteq G$.
	When $b$ is basic \Cref{mainthm} can be proved directly, and it is an unpublished result of Hansen--Weinstein. Suppose that $b$ is not basic and let $P\subseteq G$ be the parabolic generated by $B$ and the centralizer of $\nu_b$.

	To prove that a space is connected it suffices to prove that a dense subset of it is connected, this allows us to replace $\Gr_{G,\mu}$ by the dense open subset $L^+P\cdot \xi^\mu$.  
	Now, by Beauville--Laszlo descent, $L^+P\cdot \xi^\mu$ gets identified with the space of modifications of $\calE^P_b$, where $\calE_b^P$ is the Harder--Narasimhan $P$-reduction of $\calE_b$.   
Moreover, on this open subset we have a factorization:
\begin{equation*}
	\calB\calL_b:L^+P\cdot \xi^\mu \xrightarrow{\calB\calL_{P,b}} \Bun_P\to \Bun_G.
\end{equation*}
Recall the following general fact.
	Let $X$ be a connected locally spatial diamond that is smooth and partially proper over $\Spa \bbC_p$. Suppose we have an open immersion $j:U\to X$ and complementary closed immersion $i:Z\to X$.  
	For $U$ to be connected, it suffices that $\on{dim}(Z)<\on{dim}(X)$ by \cite[Corollary 4.11]{Han21}.
	In our case $X=L^+P\cdot \xi^\mu$ and  $U=L^+P\cdot \xi^\mu\cap \calB\calL_b^{-1}(\Bun_G^1)$.
	An important observation is that the non-empty fibers of $\calB\calL_{P,b}$ are $\on{Aut}_{\on{Fil}}(\calE_b)$-torsors, see \Cref{lemma-fibers-have-same-dim}. In particular, they all have the same dimension. 
	Also, $\calB\calL_{P,b}$ factors through one connected component $\Bun^\kappa_P\subseteq \Bun_P$ determined by $\mu-\nu_b$.

	Let $Y=\Bun^\kappa_P\setminus \calB\calL_b^{-1}(\Bun_G^1)$.
	The second key point is that $\dim(Y)<\dim(\Bun^\kappa_P)$. 
	To prove this, we study the following diagram: 
	\begin{equation}
		\label{diagram-eisenstein}
		\begin{tikzcd}
			\Bun_P^{b_M} \ar{r} \ar{d}		& \Bun_P \ar{r} \ar{d} & \Bun_G \\
			\Bun_M^{b_M}\ar{r}	& \Bun_M
		\end{tikzcd}
	\end{equation}
	where $M$ is the Levi quotient of $P$, $b_M\in B(M)$ and the square is Cartesian. When $b_M$ is basic and $\nu_{b_M}$ is $G$-dominant, $\Bun_P^{b_M}\to \Bun_G$ is smooth and dimensions are easy to understand. 
	On the other hand, the case when $b_M$ is basic and $\nu_{b_M}$ is a non-negative sum of positive coroots, i.e.~ $\nu_{b_M}\in \bbQ_{\geq 0} \Phi_G^+$, can be understood inductively from the case where $\nu_{b_M}$ is $G$-dominant. 
	This is where $b\in B(G,\mu)$ is important. Indeed, in this case $\mu^\diamond-\nu_b$ is in $\bbQ_{\geq 0}\Phi_G^+$ and the relevant $b_M\in B(M)_{\on{basic}}$ satisfies that $\nu_{b_M}$ is also in $\bbQ_{\geq 0}\Phi_G^+$. \\

	We now explain the organization of this article. We start $\mathsection 2$ with some cohomological considerations that allow us to work with the notion of dimension in a meaningful way. Then, we make some preparations explaining the combinatorics involving the induction process that reduces the $\nu_{b_M}\in \bbQ_{\geq 0}\Phi^+_G$ case to the $G$-dominant case. Afterwards, we bound dimensions of Newton strata that arise from the diagram \ref{diagram-eisenstein}. Finally, $\mathsection 3$ is dedicated to proving \Cref{mainthm}.

	\subsection{Acknowledgements} 
	We thank Linus Hamann for explaining certain ideas that are key to this work.
	This paper was written during stays at Max-Planck-Institut für Mathematik and Universität Bonn, we are thankful for the hospitality of these institutions. The project has received funding by DFG via the Leibniz-Preis of Peter Scholze (I.G), and (J.L.) from the Max-Planck-Institut für Mathematik.

	We would also like to thank Alexander Ivanov, Louis Jaburi, Andreas Mihatsch, Michael Rapoport, Peter Scholze, Eva Viehmann and Mingjia Zhang for interesting conversations and feedback on this paper. 
	
\section{Bounding dimensions of Newton strata.}
\subsection{Dimension for stacky maps}
In the following sections we bound the dimensions of certain Artin v-stacks.
Since we do not intend to develop foundations, we will work with an ad-hoc notion of dimension. 
Let $f:X\to Y$ be a \textit{fine} morphism of Artin v-stacks \cite[Definition 1.3]{GHW22} and let $n\in \bbN$.
Let $S\to Y$ be a map with $S$ a spatial diamond, let $f_S:X_S\to S$ denote the base change, and let $\calF\in D^{\leq 0}_\et(X_S,\bbF_\ell)$.
\begin{definition}
We say that the \textit{$\ell$-cohomological dimension of $f$ is bounded by $n$}, which we abbreviate as ${\dim}_\ell(f)\leq n$ if: for all $S\to Y$ and $\calF$ as above:
\begin{equation}
f_{S,!}\calF\in D^{\leq 2n}_{\acute{e}t}(S,\bbF_\ell),
\end{equation}
and we write $\dim_\ell(X)\leq n$ when $Y=\ast$.
\end{definition}

\begin{convention}
From now on we will only consider maps of Artin v-stacks that are fine and we will not include this adjective in our statements.
\end{convention}
Actually, the stacky morphisms used in this article are all obtained as compositions of smooth maps and locally closed immersions which are all fine morphisms. 
\begin{lemma}
	\label{lemma-dimension-subaditive}
   Let $f:X\to Y$ and $g:Y\to Z$ be map of Artin v-stacks such that $\dim_\ell(f)\leq n$ and $\dim_\ell(g)\leq m$. Then $\dim_\ell(g\circ f)\leq m+n$.
\end{lemma}
\begin{proof}
	Let $S\to Z$ be a map and denote by $X_S$ and $Y_S$ the base changes.
	Let $\calF\in D^{\leq 0}_\et(X_S,\bbF_\ell)$. 
	Observe that $f_{S,!}\calF[2n] \in D^{\leq 0}_\et(Y_S,\bbF_\ell)$, 
	which implies that $g_{!,S}f_{S,!}\calF[2n]\in D^{\leq 2m}(S,\bbF_\ell)$. It follows that $\dim_\ell(g\circ f)\leq n+m$.
\end{proof}
\begin{lemma}
	\label{lemma-fibers-bound-dimension}
	Let $f:X\to Y$ be a map of Artin v-stacks. Suppose that for any $s:\Spa(C,C^+)\to X$ the fibers satisfy $\dim_\ell(X_s)\leq n$. Then $\dim_\ell(f)\leq n$.
\end{lemma}
\begin{proof}
	This follows from \cite[Theorem 1.9.(2)]{Sch17}, \cite[Theorem 1.4.(4)]{GHW22}, since $\calF\in D^{\leq 2n}_\et(S,\bbF_\ell)$ can be checked on geometric point.
\end{proof}
\begin{lemma}
    Let $f:X\to Y$ be a surjective $\ell$-cohomologically smooth map of Artin v-stacks with constant $\ell$-dimension $d$. 
    Let $g:Y\to Z$ be a map of Artin v-stacks. Then $\dim_\ell(g)\leq n$ if and only if $\dim_\ell(g\circ f)\leq n+d$.
\end{lemma}
\begin{proof}
	To bound $\dim_\ell(g\circ f)$ it suffices by \Cref{lemma-dimension-subaditive} to prove $\dim_\ell(f)\leq d$.
	It suffices to prove that $\on{RHom}(f_{S,!}\calF, \calG)=0$ for every map $S\to Y$, every object $\calG\in D_\et^{\geq 2d+1}(S,\bbF_\ell)$ and every object $\calF\in D^{\leq 0}_\et(X,\bbF_\ell)$.
	By adjunction, we may prove $\on{RHom}(\calF, f^!_S\calG)=0$ instead.
	Now, by $\ell$-cohomological smoothness $f^!\calG=f^*\calG \otimes f^!\bbF_\ell$ and $f^!\bbF_\ell$ is an invertible object in $D_\et(X,\bbF_\ell)$ concentrated in degree $-2d$.
	In particular, $f_S^!\calG\in D_\et^{\geq 1}(X,\bbF_\ell)$ while $\calF\in D_\et^{\leq 0}(X_S,\bbF_\ell)$. 

	To prove $\dim_\ell(g)\leq n$, let $S\to Z$ a map with $S$ a spatial diamond, let $\calF\in D^{\leq 0}_\et(Y_S,\bbF_\ell)$ and let $\calG\in D_\et^{\geq 2n+1}(S,\bbF_\ell)$.
As above, it suffices to prove: 
\begin{equation}
\on{RHom}(\calF, g_S^!\calG)=0
\end{equation}
In other words, we wish to prove that $g_S^!\calG\in D^{\geq 1}_\et(Y_S,\bbF_\ell)$, for all $\calG\in  D_\et^{\geq 2n+1}(S,\bbF_\ell)$. 
This can be verified on geometric points so we may show 
\begin{equation}
	\label{random-equation}
	f_S^*g_S^!\calG\in D^{\geq 1}_\et(X_S,\bbF_\ell)
\end{equation}
instead, since $f_S$ is surjective.
 By smoothness, $f^!_S\bbF_\ell\in D^{-2d}_\et(X_S,\bbF_\ell)$ is an invertible object and $f_S^*g_S^!\calG=f_S^!g_S^!\calG \otimes (f_S^!\bbF_\ell)^{-1}$.
Since, by assumption $f_S^!g_S^!\calG \in D^{\geq 1-{2d}}_\et(X_S,\bbF_\ell)$, we can verify that \ref{random-equation} holds.
\end{proof}
\begin{lemma}
	\label{lemma-filtering-dimension}
	Let $f:X\to Y$ be a map of Artin v-stacks. Let $i:Z\to X$ be a closed immersion and let $j:U\to X$ denote the complementary open immersion. Suppose that $\dim_\ell(i\circ f)\leq n$ and that $\dim_\ell(j\circ f)\leq n$, then $\dim_\ell(f)\leq n$. Conversely if $\dim_\ell(f)\leq n$ then $\dim_\ell(i\circ f)\leq n$ and $\dim_\ell(j\circ f)\leq n$. 
\end{lemma}
\begin{proof}
	Notice that the fibers of $j$ and $i$ are $0$-dimensional. By \Cref{lemma-dimension-subaditive} the second claim follows. 
	For the first claim, let $\calF\in D_\et^{\leq 0}(X,\bbF_\ell)$, and consider the following distinguished triangle:	
	\begin{equation}
		f_!j_!j^*\calF\to f_!\calF \to f_!i_*i^*\calF \to f_!j_!j^*\calF[1]
	\end{equation}
	We may pass to geometric fibers, where one of the terms vanish.
\end{proof}

\subsection{Averages of coweights}
Let $G$ be a quasi-split reductive group over $\bbQ_p$ and let $T\subset B\subset G$ be a pair consisting of a maximal torus that is maximally $\bbQ_p$-split and a Borel both defined over $\bbQ_p$. 
Let $\Phi_G$ be the absolute root system of $G$ with respect to $T$ and $\Delta_G$ the basis of positive simple absolute roots with respect to $B$.
We let $X_*(T)$ denote the set of geometric cocharacters and denote by $X_*(T)_\bbQ$ and $X_*(T)_\bbR$ the resulting rational vector space.
We use the symbol $M$ to denote a standard Levi of $G$ defined over $\bbQ_p$, and by $\Delta_M$ the induced base of positive simple roots.

\begin{definition}
	We say that $\nu \in X_*(T)_\bbQ$ is $M$\textit{-dominant} (resp. $M$\textit{-central}) if $\langle \alpha, \nu \rangle \geq 0$ (resp. $\langle \alpha, \nu \rangle=0$) for all $\alpha \in \Delta_M$ and denote by $X_*(T)_\bbQ^{+_M}$ the convex set of $M$-dominant vectors in $X_*(T)_\bbQ$. 
\end{definition}

Following \cite{Sch22}, we now define the so called $M$\textit{-average} of $\nu$:
\begin{equation}
\mathrm{av}_M(\nu)=\frac{1}{\vert W_M \rvert}\sum_{w \in W_M } w\nu
\end{equation} 
where $W_M$ denotes the absolute Weyl group of $M$.
\begin{lemma}
The $M$-average $\mathrm{av}_M(\nu)$ is the unique $M$-central $\mu \in X_*(T)_\bbQ$ whose difference $\mu-\nu$ is spanned by $\Delta_M^\vee$.
\end{lemma}

\begin{proof}
	Notice that $\mathrm{av}_M(\nu)$ is $W_M$-invariant by definition. Also, a vector is $W_M$-invariant if and only if it is $M$-central.
\end{proof}

It also follows that $\langle 2\rho_G-2\rho_M, \nu\rangle=\langle 2\rho_G-2\rho_M, \on{av}_M(\nu)\rangle$. We study how averaging interacts with the notion of positivity presented below.
	
\begin{definition}
	We say that $\nu \in X_*(T)_\bbQ$ is \textit{non-negative} if it belongs to the convex hull of $X_*(Z_G)_\bbQ$ and $\bbQ_{\geq 0}\alpha^\vee$, where $Z_G$ is the center of $G$ and $\alpha$ runs over $\Delta_G$. The convex set of non-negative vectors 
	is denoted by $X_*(T)_\bbQ^{\geq 0}$.
\end{definition}

Our definition above corresponds to the inequality $\nu_{\mathrm{ad}} \geq 0$ in the usual Bruhat order of $X_*(T_{\mathrm{ad}})$, where $T_{\mathrm{ad}}$ denotes the image of $T$ in the adjoint group $G_{\mathrm{ad}}$ of $G$. A dominant vector is necessarily non-negative, but the converse rarely ever holds. In the following, we note that averaging preserves non-negativity, compare with \cite[Lemma 3.1]{Sch22}.

\begin{proposition}
	\label{average-preserves-positivity}
	The function $\mathrm{av}_M$ preserves $ X_*(T)_\bbR^{\geq 0}$.
\end{proposition}

\begin{proof}
It suffices to see that it preserves $X_*(Z_G)_\bbQ$ and $\bbQ_{\geq 0}\alpha^\vee$. This is clear for $M$-central coweights, so it suffices to consider $\mathrm{av}_M(\alpha^\vee)$ for $\alpha \in \Delta_G \setminus \Delta_M$. But then $w\alpha^\vee$ is a positive coroot for all $w \in W_M$, thereby finishing the proof.
\end{proof}

\begin{remark}
	If $G=\mathrm{GL}_n$, we may interpret $\nu$ as a polygon and its non-negativity as meaning the polygon never crosses the straight line connecting its extremities. The vector $\mathrm{av}_M(\nu)$ corresponds to connecting vertices according to a partition of $n$. In this case, it is visually clear that this partial average polygon lies above the total average polygon, since we started with a non-negative one. 
\end{remark}
As a corollary, we get the following technical result that is relevant in the next subsection:

\begin{lemma}
	\label{combinatorics-positive-vs-dominant}
	Let $\nu \in X_*(T)^{\geq 0}_\bbQ$ be invariant under $\Gamma$ and $M$-central. There is a sequence of standard Levi subgroups $M=M_0 \subset\dots \subset M_i \subset \dots \subset M_k=G $ defined over $\bbQ_p$ and also of $\Gamma$-invariant vectors $\nu=\nu_0, \dots, \nu_i, \dots \nu_k=\on{av}_G(\nu)$ in $X_*(T)^{\geq 0}_\bbQ$ such that the following properties hold
	\begin{enumerate}
		\item $\nu_j=\mathrm{av}_{M_j}(\nu_i)$ for $j\geq i$.
		\item $\nu_i$ is $M_{i+1}$-dominant.
	\end{enumerate}
\end{lemma}

\begin{proof}
	Suppose $\langle \alpha, \nu\rangle \leq 0$ for all $\alpha \in \Delta_G\setminus \Delta_M$. 
	Since $\langle \alpha, \nu \rangle =0$ for $\alpha \in \Delta_M$ by hypothesis, we also get $\langle \rho_G, \nu\rangle \leq 0$. On the other hand, the convex hull of $X_*(Z_G)_\bbQ$ and $\bbQ_{\geq 0}\alpha^\vee$ for all $\alpha \in \Delta_{G}$ pairs non-negatively with the strictly dominant weight $\rho_{G}$, and it vanishes exactly on $G$-central elements. Therefore, the only possibility would be $M=G$, in which case $k=0$. 
	
	Otherwise, there exists some $\alpha \in\Delta_{G}\setminus \Delta_M$ such that $\langle \alpha, \nu \rangle > 0$. By $\Gamma$-invariance, this holds for its entire $\Gamma$-orbit.
	Now let $L$ be the standard Levi defined over $\bbQ_p$ with $\Delta_L=\Delta_M \cup \Gamma\alpha$ and consider $\on{av}_L(\nu)$. By \Cref{average-preserves-positivity} $\mathrm{av}_L(\nu)$ is non-negative and $L$-central, which finishes the proof of the lemma by induction on the cardinality of $\Delta_G \setminus \Delta_M$.
\end{proof}

\subsection{Newton strata in $\Bun_P$.}
We let $\bbM$ denote the set of standard Levi subgroups of $G$ containing $T$. 
Let $B(\bbM)$ denote the set of pairs $\{(M,b_M)\}$ where $M\in \bbM$ and $b_M\in B(M)$. 
For all $M\in \bbM$, we have $B(M)\subset B(\bbM)$.

Fix $b\in B(\bbM)$ with $b=(M,b_M)$ and let $P=MB$ denote the standard parabolic containing $B$ and with standard Levi $M$.
We let $\nu_b\in (X_*(T)\otimes \bbQ)^\Gamma$ denote the $M$-dominant Newton point of $b_M$. 
For $M\subseteq L$ we let $P_L:=P\cap L$ and define $\Bun_{P_L}^b$ by the following diagram with Cartesian square: 
\begin{equation}
	\begin{tikzcd}
		\Bun_{P_L}^b \ar{r} \ar{d} & \Bun_{P_L} \ar{d} \ar{r} & \Bun_L \\	
		\Bun^{b_M}_M \ar{r} & \Bun_M
	\end{tikzcd}
\end{equation}
\begin{theorem}[Hamann]
	\label{Mb-is-smooth}
	The map of Artin v-stacks $\Bun_{P_L}\to \Bun_M$ is $\ell$-cohomologically smooth. In particular, $\Bun_{P_L}^b \to \Bun^{b_M}_M$ is $\ell$-cohomologically smooth. Moreover, the later map is of relative $\ell$-dimension $\langle 2\rho_L-2\rho_M,\nu_b \rangle$.
\end{theorem}

\begin{proof}
	This follows from \cite[Proposition 3.16, Proposition 4.7]{Ham22}. 
\end{proof}
When $M\subseteq L$ we let $i_L(b)$ denote the pair $(L,b_M)$ and $a_L(b)$ denote the pair $(L,b_L)$ where $b_L$ is the unique basic element in $B(L)$ with the same image image under the Kottwitz map, i.e.~ with $\kappa_L(b_{L})=\kappa_L(b_M)$. 
One verifies that $\mathrm{av}_L(\nu_b)=\nu_{a_L(b)}$, and consequently: \[\langle 2\rho_L-2\rho_M,\nu_b\rangle= \langle 2\rho_L-2\rho_M,\nu_{a_L(b)}\rangle.\]

\begin{definition}
	\label{definition-postivie-dom-etc}	
	We say that $b\in B(\bbM)$ is \textit{basic} if $b_M\in B(M)$ is basic. 
	We say that $b\in B(\bbM)$ is \textit{dominant} if $\nu_b$ is $G$-dominant.
	We say that $b\in B(\bbM)$ is \textit{non-negative} if $\nu_b$ lies in the monoid generated by $X_*(Z_G)_\bbQ^\Gamma$ and $\bbQ_{\geq 0} \alpha^\vee$ for every $\alpha \in \Delta_G$. 
\end{definition}

Notice that if $b$ is basic and dominant then $\Bun_P^b=\calM_b$, the Fargues--Scholze chart attached to $i_G(b)\in B(G)$ \cite[\S V.3]{FS21}. 
Also, if $b$ is basic and anti-dominant $\Bun_P^b\to \Bun_G$ induces an isomorphism $\Bun_P^b\cong \Bun_G^b$.

Let $g\in B(\bbM)$ with $g=(L_g,g_L)$ and $M\subseteq L_g$. 
Let $L_g\subseteq L$ and let 
\begin{equation}
	\label{doubly-indexed-strata}
\Bun_{P_L}^{(b,g)}:=\Delta^{-1}_{M,L_g}(\Bun^{b_M}_M\times \Bun^{g_L}_{L_g})\subseteq \Bun_{P_L}
\end{equation}
Here $\Delta_{M,L_g}:\Bun_{P_L}\to \Bun_M\times \Bun_{L_g}$ is the composition of $\Bun_{P_L}\to \Bun_{P_{L_g}}$ and the diagonal map $\Bun_{P_{L_g}}\to \Bun_M\times \Bun_{L_g}$.
\begin{proposition}
	\label{proposition-bounding-dimensions-Mb}
	If $b\in B(\bbM)$ is basic and non-negative, then $\Bun_P^b$ contains an open subspace $\calT_b\subset \Bun_P^b$ such that $f_b:\calT_b\to \mathrm{Bun}_G$ is $\ell$-cohomologically smooth of relative dimension $\langle 2 \rho_G-2\rho_M,\nu_b\rangle$. Moreover, $f_b$ factors through $\Bun^{a_G(b)}_G$ and $\dim_\ell(\Bun_P^b\setminus \calT_b)< \langle 2 \rho_G-2\rho_M,\nu_b\rangle$.   
\end{proposition}
\begin{proof}
	We do this by induction on the cardinality of $\Delta_G \setminus \Delta_M$. If $b$ is basic and dominant then $\Bun_P^b\to \Bun_G$ is smooth by \cite[Theorem V.3.7]{FS21} and with notation as in \cref{doubly-indexed-strata} $\calT_b=\Bun_P^{(b,a_G(b))}$ satisfies the desired properties.
We may choose $L=M_1$ as in the statement of \Cref{combinatorics-positive-vs-dominant}. 
We let $Q \subset G$ denote the parabolic generated by $L$ and $B$ and we let $P_L=L\cap P$.
We have the following commutative diagram with Cartesian squares: 
\begin{equation}
	\begin{tikzcd}
		\Bun_P^b \ar{r} \ar{d} & \Bun_P \ar{r} \ar{d} & \Bun_Q \ar{r} \ar{d} & \Bun_G  \\
		\Bun_{P_L}^{b}	\ar{r} \ar{d}& \Bun_{P_L} \ar{r} \ar{d}& \Bun_L  \\
		\Bun^{b_M}_M	\ar{r}  & \Bun_M
	\end{tikzcd}
\end{equation}
After pullback by $\Bun_L^{b_L}\to \Bun_L$, and by induction, we get a commutative diagram in which $\calT_b$ is defined so that all squares are Cartesian:
\begin{equation}
	\begin{tikzcd}
		\calT_b \ar{r} \ar{d} & \Bun_{P}^{(b,a_L(b))} \ar{r} \ar{d} & \Bun_P^b \ar{r} \ar{d} & \Bun_{P_L}^b \ar{d} \\	
		\calT_{a_L(b)} \ar{r} & \Bun_Q^{a_L(b)} \ar{r} & \Bun_Q \ar{r} & \Bun_L 
	\end{tikzcd}
\end{equation}
By induction, the map $\calT_{a_L(b)}\to \Bun_G$ is $\ell$-cohomologically smooth 
and $\calT_b\to \calT_{a_L(b)}$ is also $\ell$-cohomologically smooth, so the same holds for their composition. The claim on dimensions follow since $\Bun_G \to \ast$ is $\ell$-smooth of dimension $0$ and $\Bun_P^b \to \ast$ is $\ell$-smooth of dimension $\langle 2\rho_G-2\rho_M, \nu_b\rangle$.

For the second claim, let $g\in B(L)$ be in the image of $\Bun_{P_L}^b$. We get a smooth map $\Bun_P^{(b,g)}\to \Bun_Q^g$ of $\ell$-dimension $\langle 2\rho_L- 2\rho_M,\nu_b \rangle$. 
By \Cref{Mb-is-smooth}, the map $\Bun_Q^g\to \Bun^g_L$ is smooth and it has $\ell$-dimension $\langle 2\rho_G- 2\rho_L,\nu_g \rangle$. 
Now, $\langle 2 \rho_G-2\rho_L,\nu_{g}\rangle=\langle 2 \rho_G-2\rho_L,\nu_{b}\rangle$, since $\kappa_L(b)=\kappa_L(g)$. In particular, $\Bun_P^{(b,g)}\to \Bun_Q^g\to \Bun_L^g$ is smooth of relative dimension $\langle2\rho_G-2\rho_M,\nu_b\rangle$. 
Now, when $g\neq a_L(b)$, $\dim_\ell(\Bun^g_L)<0$, so that by \Cref{lemma-dimension-subaditive} and \Cref{lemma-filtering-dimension} $\dim_\ell(\Bun_P^{(b,g)})<\langle 2 \rho_G-2\rho_M,\nu_b\rangle$ and $\dim_\ell(\Bun_P^b\setminus \Bun_P^{(b,a_L(b))}) <\langle 2 \rho_G-2\rho_M,\nu_b\rangle$.
By induction, $\dim_\ell(\Bun_P^{(b,a_L(b))}\setminus \calT_b)<\langle 2 \rho_G-2\rho_M,\nu_b\rangle$ since $\Bun_P^{(b,a_L(b))}\setminus \calT_b\to \Bun_Q^{a_L(b)}\setminus \calT_{a_L(b)}$ is smooth.
By \Cref{lemma-filtering-dimension}, $\dim_\ell(\Bun_P^{b}\setminus \calT_b)<\langle 2 \rho_G-2\rho_M,\nu_b\rangle$, since we have an open and closed decomposition $\Bun_P^{b}\setminus \calT_b=(\Bun_P^{(b,a_L(b))}\setminus \calT_b) \cup  (\Bun_P^{b}\setminus \Bun_P^{(b,a_L(b))})$. 
\end{proof}

\section{$\Gr^b_{G,\mu}$ is connected}
Contrary to the previous section we will momentarily not assume that $G$ is quasi-split.
Fix $C$ an algebraically closed non-Archimedean field extension of $\breve{E}$ and recall the Beauville--Laszlo map from the introduction
\begin{equation}
	\calB\calL_b:\Gr_{G,\mu}\to \Bun_{G},
\end{equation}
where we base change the affine Grassmannian to $\Spd C$.
Observe that $\calB\calL_b$ factors through the unique connected component of $\Bun_{G}$ parametrized by $\mu^\natural-\kappa_G(b)\in \pi_1(G)_\Gamma$.
We formulate \Cref{mainthm} as follows: 

\begin{theorem}
	If $b\in B(G,\mu)$, then $\Gr^{b}_{G,\mu}$ is dense in $\Gr_{G,\mu}$ and connected. 	
\end{theorem}
Without loss of generality we may assume that $G$ is adjoint. 
Moreover, we may replace $G$ by its quasi-split inner form $G^*$, which is now a pure inner form by adjointness of $G$. In total, we may assume that $G$ is quasi-split, at the expense of having to prove the more general \Cref{thm-quasi-split} below.
 
Let us recall the setup. Let $T\subset B\subset G=G^*$ be as in the previous section.
We define an element $\mu^\diamond\in X_*(T)^\Gamma_\bbQ$ given by the formula:
\begin{equation}
\mu^\diamond\coloneqq\frac{1}{[\Gamma:\Gamma_\mu]}\sum_{\gamma\in \Gamma/\Gamma_\mu} \gamma(\mu),
\end{equation}
where $\Gamma_\mu$ denotes the stabilizer of $\mu$ for the $\Gamma$-action. Notice that $\langle 2\rho_G, \mu^\diamond \rangle= \langle 2\rho_G, \mu \rangle$, because $\rho_G$ is $\Gamma$-invariant.

Let $A_Z(G,\mu)\subset B(G)$ be the set of acceptable elements modulo center, i.e.~ for which $\mu^\diamond-\nu_b$ is non-negative as in \Cref{definition-postivie-dom-etc}. This is related to the notion of acceptable elements $A(G,\mu)$ of \cite[Definition 2.3]{RV14}, in the sense that $A_Z(G,\mu)$ equals the pre-image of $A(G_\mathrm{ad},\mu_\mathrm{ad})$ along $B(G)\to B(G_\mathrm{ad})$.

If $b\in B(M)$, we let $\mathrm{d}_{\mu,b}^M$ denote the unique basic element in $B(M)$ such that $\kappa_M(\mathrm{d}_{\mu,b}^M)=\mu^\natural-\kappa_M(b)$.
When $M=G$ we simply write $b_\mu$ for $\mathrm{d}_{\mu,b}^G$.
Let $d=\dim_\ell(\Gr_{G,\mu})=\langle 2\rho_G, \mu\rangle$ and let $\Gr_{G,\mu}^{(g,b)}:=\calB\calL_b^{-1}(\Bun^{g}_G)\subset \Gr_{G,\mu}$.
For example, $\Gr_{G,\mu}^b=\Gr_{G,\mu}^{(1,b)}$.
\begin{theorem}
	\label{thm-quasi-split}
	If $b\in A_Z(G,\mu)$, then $\Gr_{G,\mu}^{(b_\mu,b)}$ is dense in $\Gr_{G,\mu}$ and connected.
\end{theorem}
\begin{proof}
To prove that $\Gr^{(b_\mu,b)}_\mu$ is dense and connected, it suffices to prove that $\dim_\ell(\Gr^{(g,b)}_\mu)<d$ for all $g\in B(G)$ with $g\neq b_\mu$. 
We consider the Schubert cell $\Gr_{G,\mu}^\circ\subset \Gr_{G,\mu}$. 
Since $\dim_\ell(\Gr_{G,\mu}\setminus \Gr_{G,\mu}^\circ)<d$ it suffices to prove that $\dim_\ell(\Gr_{G,\mu}^{(g,b)}\cap \Gr_{G,\mu}^\circ)<d$.
If $b$ is basic, $\calB\calL_b:[\underline{G(\bbQ_p)}\backslash \Gr_{G,\mu}^\circ]\to \Bun_G$ is smooth of relative dimension $d$ \cite{FS21}. 
In particular, $\dim_\ell(\Gr_{G,\mu}^{\circ,(g,b)})=d+\dim_\ell(\Bun_G^g)$. 
Now, $b_\mu$ is the unique basic element in the image of $\calB\calL_b$ and for non-basic elements $\dim_\ell(\Bun_G^g)<0$. This finishes the proof in this case.

Suppose now that $b$ is not basic, let $M$ denote the centralizer of $\nu_b$, let $b_M$ denote the unique element in $B(M)$ mapping to $b$ whose Newton point is $G$-antidominant. 
Now, $\Bun_{P}^{b_M}\cong \Bun_G^b$ by our choice of $b_M$, and we let $\calE^P_b$ denote the unique $P$-reduction of $\calE_b$ determined by the image of $\Bun_{P}^{b_M}$ in $\Bun_P$.
The space of modifications of $\calE_b^P$ gets identified with $\Gr_P\subset \Gr_G$. 
We consider $\Gr_{P,\mu}^{\circ}:=L^+P \cdot \xi^\mu$, the result of intersecting $\Gr_{G,\mu}$ with the connected component of $\Gr_P$ attached to the dominant representative $\mu$.
We have a smooth map $\Gr^\circ_{P,\mu}\to \Gr^\circ_{M,\mu}$ of relative dimension $\langle 2\rho_G-2\rho_M, \mu\rangle$. 
Moreover, we have a commutative diagram: 
\begin{equation}
	\begin{tikzcd}
		\Gr^{\mathrm{d}_{\mu,b}^M}_{P,\mu}\ar{r} \ar{d}	& \Gr^\circ_{P,\mu} \ar{r} \ar{d} & \Gr^\circ_{G,\mu} \ar{d}{\calB\calL_b}  \\	
\Bun_{P}^{\mathrm{d}_{\mu,b}^M} \ar{r} \ar{d}	& 	\Bun_P \ar{d} \ar{r} & \Bun_G \\
	\Bun_M^{\mathrm{d}_{\mu,b}^M}\ar{r}	& 	\Bun_M  
	\end{tikzcd}
\end{equation}
Where $\Gr^{\mathrm{d}_{\mu,b}^M}_{P,\mu}$ is defined so that the square in the left-upper corner is Cartesian.
In particular, the upper left arrow is an open immersion. 
Since $\mathrm{d}^M_{\mu,b}\in B(M)$ is basic, we know that \begin{equation}\dim_\ell(\Gr^\circ_{P,\mu}\setminus \Gr^{\mathrm{d}_{\mu,b}^M}_{P,\mu})<d.\end{equation}
It suffices to prove that \begin{equation}\label{eqn.desig.2}\dim_\ell(\Gr^{\mathrm{d}_{\mu,b}^M}_{P,\mu}\cap \Gr_{G,\mu}^{(g,b)})<d\end{equation} for $g\neq b_\mu$.
By \Cref{proposition-bounding-dimensions-Mb},
\begin{equation}
	\dim_\ell(\Bun_{P}^{(\mathrm{d}_{\mu,b}^M,g)})< \langle 2\rho_G-2\rho_M, \nu_{\mathrm{d}_{\mu,b}^M}\rangle = \langle 2\rho_G-2\rho_M, \on{av}_M(\mu^\diamond-\nu_b)\rangle
\end{equation}
By \Cref{lemma-fibers-have-same-dim}, the geometric fibers of \begin{equation}\label{eq_mor_crucial}
\Gr^{\mathrm{d}_{\mu,b}^M}_{P,\mu}\to \Gr_{M,\mu}^\circ \times_{\Bun_M} \Bun_{P}^{\mathrm{d}_{\mu,b}^M}
\end{equation} have all dimension bounded by $\langle 2\rho_G-2\rho_M, \nu_b\rangle$. 
Consequently by \Cref{lemma-fibers-bound-dimension}, we get that \eqref{eqn.desig.2} holds. 
Indeed, $\dim_\ell(\Gr^{\mathrm{d}_{\mu,b}^M}_{P,\mu}\cap \Gr_{G,\mu}^{(g,b)})$ is bounded by the dimension of $\Gr_{M,\mu}^\circ \times_{\Bun_M} \Bun_{P}^{(\mathrm{d}_{\mu,b}^M,g)}$ and the dimension of the fiber. The former is smaller than $\langle2\rho_M,\mu\rangle + \langle 2\rho_G-2\rho_M, \on{av}_M(\mu^\diamond-\nu_b)\rangle$ and the later is $\langle 2\rho_G-2\rho_M, \nu_b\rangle$.
Moreover, 
$\langle 2\rho_G-2\rho_M, \on{av}_M(\mu^\diamond-\nu_b)\rangle=\langle 2\rho_G-2\rho_M, \mu^\diamond-\nu_b\rangle$ and $\langle 2\rho_G-2\rho_M, \mu^\diamond\rangle=\langle2\rho_G-2\rho_M, \mu\rangle$.
\end{proof}

\begin{lemma}
	\label{lemma-fibers-have-same-dim}
	 The geometric fibers of \eqref{eq_mor_crucial} are either $\on{Aut}^{\mathrm{unip}}_{\on{Fil}}(\calE_b)$-torsors or empty. Their dimension is $\langle 2\rho_G-2\rho_M,\nu_b\rangle$ in the former case.
\end{lemma}

\begin{proof}
    We begin by observing that the geometric fibers of the Beauville--Laszlo map $\Gr_{P} \to \Bun_P$ are torsors on the left for the group $A^{-1}P(B_e)A$ where $A \in P(B_\dR)$ is the Beauville--Laszlo glueing data for the $P$-torsor $\calE_b^P$ \cite[Theorem 13.5.3.(2)]{SW20}. 
    Similarly, the geometric fibers of $\Gr_M \to \Bun_M$ are $A^{-1}M(B_e)A$-torsors. We deduce that the non-empty geometric fibers of $\Gr_P \to \Bun_P \times_{\Bun_M} \Gr_M $ are torsors under the group $A^{-1}U(B_e)A$.
    

    Recall that every $t\in P(B_\dR)$ has a unique expression $t=u_t\cdot m_t$ with $u\in U(B_\dR)$ and $m\in M(B_\dR)$. We claim that if $t\in P(B_\dR^+)\xi^\mu P(B_\dR^+)$ then $u_t\in U(B_\dR^+)$. This follows from the normality of $U(B_\dR^+)$ in $P(B_\dR^+)$ and from the inclusion $\xi^\mu U(B_\dR^+)\subseteq U(B_\dR^+) \xi^\mu$, which follows from the fact that $\mu$ was assumed to be dominant. 
    Consequently, if $u\in U(B_\dR)$, $x\in \Gr_{P,\mu}^\circ$ are such that $u\cdot x\in \Gr_{P,\mu}^\circ$, then we conclude that necessarily $u\in U(B_\dR^+)$. 

    This implies that 
    the non-empty geometric fibers of our map \eqref{eq_mor_crucial} form a torsor under the group $U(B_\dR^+)\cap A^{-1}U(B_e)A=\on{Aut}^{\mathrm{unip}}_{\on{Fil}}(\calE_b)$. By \cite[Proposition III.5.1]{FS21} $\dim_\ell(\on{Aut}^{\mathrm{unip}}_{\on{Fil}}(\calE_b))=\langle 2\rho_G-2\rho_M,\nu_b\rangle$, and we may conclude the same about the non-empty fibers.
\end{proof}
	\bibliography{biblio.bib}
	\bibliographystyle{alpha}
	
\end{document}